\documentclass[12pt, reqno]{amsart}
\usepackage{amssymb,amsthm,amsmath}
\usepackage[numbers,sort&compress]{natbib}
\usepackage{amsfonts}
\usepackage{mathrsfs}
\usepackage{latexsym}
\usepackage{tikz}
\usepackage{tkz-euclide}
\usepackage{float}
\usepackage{pdfsync}
\usepackage{indentfirst}
\date{today}

\usepackage[colorlinks=true, linkcolor=blue, citecolor=red]{hyperref}

\newcommand{\R}{\mathbb R}

\newtheorem*{proof*}{Proof}
\newcommand{\beq}{\begin{equation}}
	\newcommand{\eeq}{\end{equation}}
\newcommand{\ben}{\begin{eqnarray}}
	\newcommand{\een}{\end{eqnarray}}
\newcommand{\beno}{\begin{eqnarray*}}
	\newcommand{\eeno}{\end{eqnarray*}}
\numberwithin{equation}{section} 
\newtheorem{Thm}{Theorem}[section] 
\newtheorem{Lem}[Thm]{Lemma} 

\newtheorem{Prop}[Thm]{Proposition}
\newtheorem{Rem}[Thm]{Remark}

\numberwithin{equation}{section}

\allowdisplaybreaks

\begin{document}
	\title[Critical mass threshold for the 2D PKSNS system]{\bf  Critical mass threshold for the 2D Patlak-Keller-Segel-Navier-Stokes system}
	
	\author{Wendong~Wang}
	\address[Wendong~Wang]{School of Mathematical Sciences, Dalian University of Technology, Dalian, 116024,  China}
	\email{wendong@dlut.edu.cn}

\author{Dongyi Wei}
	\address[Dongyi Wei]{School of  Mathematical Sciences and LMAM, Peking University, Beijing 100871, China }
	\email{jnwdyi@pku.edu.cn}

	\author{Zhifei Zhang }
	\address[Zhifei Zhang]{School of  Mathematical Sciences and LMAM, Peking University, Beijing 100871, China }
	\email{zfzhang@math.pku.edu.cn}

	\date{\today}
	\maketitle

	\begin{abstract}In this paper, we investigate critical mass threshold for the  Patlak-Keller-Segel-Navier-Stokes system on the two-dimensional whole space and obtain global existence of strong solutions if the initial mass  is less than or equal to $8\pi$, regardless of the initial norm of the velocity. One new observation is that the local mass of the density function rearrangement satisfies a good inequality that is independent of velocity; and then an improved maximum principle is applied by choosing a nice auxiliary function.
	\end{abstract}

	
	\setcounter{equation}{0}

	\section{Introduction}
	
	In this paper, we investigate  the critical mass threshold of the following two-dimensional parabolic-elliptic Patlak-Keller-Segel (PKS) system coupled with Navier-Stokes (NS) equations in $ \mathbb{R}^{2}\times(0,T)$:
\begin{equation}\label{ksns}
	\left\{
	\begin{array}{lr}
		\partial_tn+u\cdot\nabla n=\triangle n-\nabla\cdot(n\nabla c), \\
		\triangle c+n=0, \\
		\partial_tu+u\cdot\nabla u+\nabla \pi=\triangle u+n\nabla \phi, \\
		\nabla\cdot u=0,
	\end{array}
	\right.
\end{equation}
along with initial conditions
$$(n,u)\big|_{t=0}=(n_{\rm in},v_{\rm in}),$$
where $n$ represents the cell density, $c$ denotes the chemoattractant density, and $u$ denotes the velocity of fluid. In addition, $\pi$ is the pressure and $\phi$ is the given potential function.

When $u=0$, $\pi=0$ and $\phi=0$, the system (\ref{ksns}) is reduced to the following classical parabolic-elliptic PKS system:
\begin{equation}\label{eq:pks}
	\left\{
	\begin{array}{lr}
		\partial_tn=\triangle n-\nabla\cdot(n\nabla c), \\
		\triangle c+n=0.
	\end{array}
	\right.
\end{equation}
The  system, jointly developed by Patlak \cite{Patlak1}, Keller and Segel \cite{Keller1},  serves as a classical mathematical framework for modeling aggregation behavior driven by both random motion and chemotaxis. There has been much progress in the study of the well-posedness of this system (\ref{eq:pks})  in $\mathbb{R}^d$, and we will briefly list some of it.

It is well known that solutions to the one-dimensional PKS system are globally well-posed. However, in spatial dimensions $d \ge 2$, solutions to the classical PKS system (\ref{eq:pks}) may develop singularities in finite time. In the two-dimensional setting, provided the integrability condition $(1 + |x|^2 + |\ln n|)n \in L^\infty([0, T), L^1(\mathbb{R}^2))$ holds, the solution satisfies the following free energy inequality:
\[
F[n(0)] \ge F[n(t)] + \int_{0}^{t} \int_{\mathbb{R}^2} n(x, s)|\nabla \ln n(x, s) - \nabla c(x, s)|^2 dx ds,
\]
for a.e. $t \in (0, T)$, where the free energy functional $F(n)$ is defined by
\[
F[n] = \int_{\mathbb{R}^2} n \ln n dx - \frac{1}{2} \int_{\mathbb{R}^2} n c dx, \quad c(x, s) = -\frac{1}{2\pi} \int_{\mathbb{R}^2} \ln |x - y|n(y,s)dy.
\]
This free energy functional was originally introduced in the context of chemotaxis models by Nagai, Senba, and Yoshida \cite{NSY1997}.
Blanchet, Dolbeault, and Perthame \cite{BDP2006} established the free energy inequality and identified a critical mass threshold $M_c = 8\pi$, where they proved that free energy solutions of \eqref{eq:pks} exist globally for initial data satisfying $n_{in} \in L^1_+(\mathbb{R}^2, (1+|x|^2)dx)$ and $n_{in} \log n_{in} \in L^1(\mathbb{R}^2)$, provided the total mass $M := \|n_{in}\|_{L^1} < 8\pi$. These results were previously announced in \cite{DP2004} and their proof relies on the free energy inequality combined with the logarithmic Hardy-Littlewood-Sobolev inequality. Another alternative proof, based on ideas from \cite{D1991}, can be found in \cite{CCLW2012}. 
Global existence for the critical mass case, $M = 8\pi$, was established by Blanchet, Carrillo, and Masmoudi \cite{BCM2008} (see also \cite{BKLN2006} for the radially symmetric case). Conversely, in the super-critical regime $M > 8\pi$, solutions are known to blow up in finite time. This is a consequence of the virial identity:
\begin{equation}\label{eq:blow-up}
\frac{d}{dt} \int_{\mathbb{R}^2} |x|^2 n(x, t)dx = 4M \left( 1 - \frac{M}{8\pi} \right); 
\end{equation}
see, for example, \cite{DS2009} or \cite{BDP2006}.
In the aforementioned works, the finite second moment assumption $|x|^2 n_{in} \in L^1(\mathbb{R}^2)$ plays a crucial role. Recently, Wei \cite{Wei2018} established global well-posedness for $M \leq 8\pi$ by assuming only $n_{in} \in L^1(\mathbb{R}^2)$, thereby removing the extra moment constraints. The similar results also hold for the parabolic-parabolic form of (\ref{eq:pks}) ($\triangle c$ is replaced by $\triangle c-\partial_t c$), which was obtained by  Hosono in \cite{H2026} with the help of a reconstructed Lyapunov functional.
For further results on this topic, we refer the reader to \cite{Calvez1,BDDM2023,CGMN2022,Schweyer1} and references therein. Finally, in spatial dimensions $d \ge 3$, finite-time blow-up may occur for arbitrarily small initial mass (see \cite{Na2000, SW2019, winkler1,DDDMW,TW2017} and related works).

It's interesting that whether Wei's result can be extended to the PKS system coupling with Navier-Stokes equations \eqref{ksns}. Gong and He in \cite{GH2021} considered the case of $\phi\equiv c$ and proved that classical solutions exist for any finite time by assuming that the initial mass $M<8\pi$ and $n_{in}(1+|x|^2)\in L^1(\mathbb{R}^2)$. Recently,  Lai, Wei and Zhou in \cite{LWZ2023} improved this result and obtained the critical mass threshold $M\leq 8\pi$ when $\phi\equiv c$, that is the solutions are global if the initial mass $\int_{\mathbb{R}^2}n_{in}dx\leq 8\pi$ along with $n_{in}\log(n_{in}), n_{in}\log(1+|x|)\in L^1(\mathbb{R}^2)$ by using a good cancellation of $n\nabla c.$ It's still an open question whether the same conclusion holds for general the potential function $\phi.$ Here we answer this  question.

 Assume that $\phi\in L^\infty\left([0,\infty); \dot{W}^{1,\infty}(\mathbb{R}^{2})\right)$, and  we say that $$0\leq n \in C_w([0, T), L^1(\mathbb{R}^2))\bigcap L^\infty_{loc}((0, T), L^\infty(\mathbb{R}^2))$$
 along with $u\in L^\infty([0, T), H^1(\mathbb{R}^2))$ is a strong solution to (\ref{ksns}) with the initial data $0\leq n_{in}\in L^1(\mathbb{R}^{2})$ and $u_{in}\in W^{1,2}(\mathbb{R}^{2})$ if \\
 (i) for all test functions $\psi \in \mathcal{D}(\mathbb{R}^2)$, there holds
\begin{equation*}
\begin{aligned}
\frac{d}{dt} \int_{\mathbb{R}^2} \psi(x)n(x, t)dx = & \int_{\mathbb{R}^2} \Delta\psi(x)n(x, t)dx + \int_{\mathbb{R}^2}u\cdot \nabla\psi(x)n(x, t)dx \\
& - \frac{1}{4\pi} \int_{\mathbb{R}^2} \int_{\mathbb{R}^2} [\nabla\psi(x) - \nabla\psi(y)] \cdot \frac{x - y}{|x - y|^2} n(x, t)n(y,t)dx dy;
\end{aligned} 
\end{equation*}
(ii) $\nabla \cdot u = 0$ is true in the sense of distributions; \\
(iii) for all divergence-free test functions $\varphi \in C_\sigma^\infty(\mathbb{R}^2; \mathbb{R}^2)$ (the space of smooth, compactly supported, solenoidal vector fields), the following equality holds:
\begin{equation*}
\begin{aligned}
\frac{d}{dt} \int_{\mathbb{R}^2} u(x, t) \cdot \varphi(x) \, dx = & \,  \int_{\mathbb{R}^2} u(x, t) \cdot \Delta \varphi(x) \, dx \\
& + \int_{\mathbb{R}^2} (u(x, t) \otimes u(x, t)) : \nabla \varphi(x) \, dx+\int_{\mathbb{R}^2} n\nabla\phi \cdot \varphi(x) \, dx,
\end{aligned} 
\end{equation*}
in the sense of distributions on $(0, T)$.
Here $C_w([0, T), L^1(\mathbb{R}^2))$ is a subspace of $L^\infty([0, T), L^1(\mathbb{R}^2))$ such that $t \mapsto \int_{\mathbb{R}^2} \psi(x) n(x,t)dx$ is continuous for any $\psi \in \mathcal{D}(\mathbb{R}^2)$.


our main result is stated as follows.
		\begin{Thm}
			\label{thm:main} Assume that $\phi\in L^\infty\left([0,\infty); \dot{W}^{1,\infty}(\mathbb{R}^{2})\right)$, $0\leq n_{in}\in L^1\cap L^{\infty}(\mathbb{R}^{2})$ and $u_{in}\in W^{1,2}(\mathbb{R}^{2})$. Then there exists a global strong 
			 solution $(n, u, \pi)$ to the Keller-Segel-Navier-Stokes equations \eqref{ksns} with the initial data $(n_{in},u_{in})$ provided that the initial mass $\int_{\mathbb{R}^{2}} n_{in}dx\leq 8\pi.$
		\end{Thm}

\begin{Rem}
Theorem \ref{thm:main} generalized the result on the  Keller-Segel system in \cite{Wei2018} to the coupled Navier-Stokes system \eqref{ksns}, which is sharp since the solution will blow up due to the inequality \eqref{eq:blow-up}. 
\end{Rem}

\begin{Rem}
Theorem \ref{thm:main} also improved and generalized the main theorem of Lai, Wei and Zhou in \cite{LWZ2023}. On the one hand, we removed the redundant initial conditions of $n_{in}\log(n_{in}), n_{in}\log(1+|x|)\in L^1(\mathbb{R}^2)$, so it is not really necessary for the initial value to satisfy this condition; on the other hand, we removed the restriction on the gravitational potential function $\phi $, which can be any function belonging to $L^\infty_t\dot{W}^{1,\infty}$. When  $\phi\equiv c$ as in \cite{LWZ2023}, according to the following Theorem  \ref{thm:main2} and \eqref{ksns} one can get $\nabla c\in L^\infty$ and then the conclusion follows. 
\end{Rem}

\begin{Rem}
Our new observation is based on  the local mass of the density function rearrangement satisfies a good inequality that is independent of velocity, which is shown in Proposition \ref{eq:k's equation}, 
\beno
\frac{\partial k}{\partial t}-4\pi s\frac{\partial^2k}{\partial s^2}-k\frac{\partial k}{\partial s}\leq 0,
\eeno
where 
\beno
k(s,t)=\int_0^sn^*(\tau,t)d\tau.
\eeno
Then we proved an improved maximum principle and   applied it by choosing a nice auxiliary function.
\end{Rem}

In detail, Theorem \ref{thm:main} follows from the following a priori estimates of  the two dimensional Keller-Segel equation
with the drift term $u.$
	\begin{Thm}
			\label{thm:main2} Assume that $0\leq n_{in}\in L^1\cap L^\infty(\mathbb{R}^{2})$, $u\in L^\infty\left([0,T]; W^{1,2}(\mathbb{R}^{2})\right)$. Then there exists a global strong 
			 solution $n$  to the Keller-Segel equation in $\mathbb{R}^{2}\times [0,T]$
             \begin{equation}\label{eq:pks-u}
	\left\{
	\begin{array}{lr}
		\partial_tn+u\cdot \nabla n=\triangle n-\nabla\cdot(n\nabla c), \\
		\triangle c+n=0.
	\end{array}
	\right.
\end{equation}
             with the initial data $n_{in}$ provided that the initial mass $\int_{\mathbb{R}^{2}} n_{in}dx=M\leq 8\pi.$ Specifically, the following estimates hold uniformly:
             \begin{itemize}
 \item $M<8\pi$: \ben\label{eq:Mless8pi}\|n(t)\|_{L^\infty}\leq \frac{8\pi}{8\pi-M}\|n_{in}\|_{L^\infty},\quad 0<t<T;\een

\item $M=8\pi$: 
\ben\label{eq:M=8pi}\|n(t)\|_{L^\infty}\leq 8 \|n_{in}\|_{L^\infty} \left(1 + 4T \|n_{in}\|_{L^\infty}\right),\quad 0<t<T.\een

 \end{itemize}
		\end{Thm}
	
\par The rest of the paper is organized as follows. In Section 2  we finish the proof of Theorem \ref{thm:main} under the assumption of  Theorem \ref{thm:main2}.   In section 3, 
the improved maximum principle is applied to prove the uniform boundedness of the  density function by choosing a suitable auxiliary function and we complete the proof of Theorem \ref{thm:main2}.
In Section 4 we show the local mass of the density function rearrangement satisfies a good inequality that is independent of velocity. In Section 5 we prove an improved maximum principle.  At last, the last section is devoted to proving the local wellposed result of the system \eqref{ksns}.

\section{Proof of Theorem \ref{thm:main}}
In this section, assuming that Theorem \ref{thm:main2} holds, we complete the proof of Theorem \ref{thm:main}.

\begin{proof}[Proof of Theorem \ref{thm:main}] By the local well-posed result of Theorem \ref{thm:local},  there exists a local strong solution to the system \eqref{ksns}. Standard bootstrapping arguments allow us to upgrade this mild solution to a strong solution for $t > 0$. 
Let $T$ be a possible blow-up time. We have $n\geq 0$ due to the maximum principle and 
$\|n(t)\|_{L^1}=\int_{\R^2}n(x,t)dx=\|n_0\|_{L^1}$ by the first equation of \eqref{ksns}, thus $n\in  L^\infty(0,T;L^1(\mathbb{R}^2))$. 
As $n_{in}\in L^1(\mathbb{R}^2)\bigcap L^\infty(\mathbb{R}^2)$ and $u\in  L^\infty(0,T';H^1(\mathbb{R}^2))$ for all $0<T'<T$.
If $T<\infty$, then by Theorem \ref{thm:main2} we have $\|n(t)\|_{L^\infty}\leq C(1+T')$ for all $0<t<T'<T$. 

Hence $n\in  L^\infty(0,T;L^1\cap L^\infty(\mathbb{R}^2))$.
Let $\omega=\partial_yu_1-\partial_xu_2=\nabla^T\cdot u$ be the vorticity of the velocity $u$, which satisfies the following equation due to \eqref{ksns}:
\ben
\partial_t \omega+u\cdot\nabla \omega-\triangle \omega=\nabla^T\cdot(n\nabla\phi)\een
which implies that
\beno
\frac{d}{dt}\|\omega\|_{L^2(\mathbb{R}^2)}^2\leq \|n\nabla\phi\|_{L^2(\mathbb{R}^2)}^2\leq C, \quad 0<t<T
\eeno
due to $n\in  L^\infty(0,T;L^1\cap L^\infty(\mathbb{R}^2))$. Similarly by the third equation of \eqref{ksns} we have 
$\frac{d}{dt}\|u\|_{L^2(\mathbb{R}^2)}\leq \|n\nabla\phi\|_{L^2(\mathbb{R}^2)}\leq C$. Thus $u\in  L^\infty(0,T;H^1(\mathbb{R}^2))$, which is a contradiction by Remark \ref{6.2}. Thus we must have $T=+\infty$.
\end{proof}

	\section{Proof of Theorem \ref{thm:main2}}

In this section, we will state the local mass of the density function rearrangement $k$ satisfies a good inequality that is independent of velocity, and establish a maximum principle. 
By using the two important properties, our main objective is to complete the proof of Theorem \ref{thm:main2}.

Let $T$ be the possible blow-up time, and $s\geq 0.$ Define 
\beno
\mu(\sigma,t)=|\{x\in \mathbb{R}^2; n(x,t)>\sigma\}|\eeno
and the  rearrangement function of $n(x,t)$ is 
\beno
n^*(s,t)=\inf_{\sigma>0}\{\sigma>0; \mu(\sigma,t)\leq s\}.
\eeno
Then it follows that
\ben\label{eq:s-n}
s=n^*(\mu(s,t), t ).
\een
Let 
\ben\label{eq:k-def}
k(s,t)=\int_0^sn^*(\tau,t)d\tau.
\een

Motivated by  Lemma 4 of Diaz and Nagai
\cite{DN1995} or Lemma 4.3 in \cite{SS2002}, we prove The following lemma, which plays an important role in the proof of the main theorem.
\begin{Prop}
    \label{eq:k's equation} Let $(n,u)$ be a local strong solution of (\ref{ksns}) in $[0,T)$ as stated  in Theorem  \ref{thm:local} and  and $k$ is defined in \eqref{eq:k-def}.
Then there holds
\ben\label{eq:mono}
\frac{\partial k}{\partial t}-4\pi s\frac{\partial^2k}{\partial s^2}-k\frac{\partial k}{\partial s}\leq 0
\een
for a.e. $(s,t)\in (0,\infty)\times (2\tau,T)$ with $\tau>0$. Moreover,
\ben\label{eq:mono-b}
k|_{s=0}=0, \frac{\partial k}{\partial s}|_{s=\infty}=0.
\een
\end{Prop}

Next we aim to prove a maximum principle on the inequality of \eqref{eq:mono-b}, which is important for the arguments of Theorem \ref{thm:main2}.

Let $\Omega$ be a domain in the two dimensional space, $Q_T=\Omega\times (0,T)$, $0<\Lambda\leq |\Omega|$, $\Omega^*=(0,\Lambda)$, and $Q_T^*=\Omega^*\times (0,T)$. $C_1(t)$ is a positive bounded function locally on $[0,T)$.
Motivated by Prop A.1 in \cite{DN1995} and  Lemma 4.3 in \cite{SS2002} we prove the following slightly different  maximum principle.

\begin{Prop}
    \label{lem:dn1995}
Let $f,g$ be functions on $\overline{Q_T^*}$ satisfying the following:\\
(i) $f,g\in L_{loc}^\infty(Q_T^*)\cap H_{loc}^1(0,T; L^2(\Omega^*))\cap L_{loc}^2(0,T; W^{2,2}(\delta,\Lambda))$ with $\delta>0$;\\
(ii) 
\beno
\left|  f(s,t)\right|+\left| \frac{\partial f}{\partial s}(s,t)\right|\leq C_1(t), \left| \frac{\partial g}{\partial s}(s,t)\right|\leq C_1(t), 0<s<\Lambda, 0<t<T;
\eeno
(iii)
\beno
\frac{\partial f}{\partial t}-4\pi s\frac{\partial^2f}{\partial s^2}-f\frac{\partial f}{\partial s}\leq \frac{\partial g}{\partial t}-4\pi s\frac{\partial^2g}{\partial s^2}-g\frac{\partial g}{\partial s}
\eeno
a.e. in  $\overline{Q_T^*}$;\\
(iv) $0=f(0,t)\leq g(0,t)$  for any $t\in (0,T)$;\\
(v)  $\frac{\partial f}{\partial s}(\Lambda, t) \leq \frac{\partial g}{\partial s}(\Lambda, t) $ or  $f(\Lambda, t) \leq g(\Lambda, t) $ with $\Lambda<\infty$, for any $t\in (0,T)$;\\
(vi) $f(s,0)\leq g(s,0)$ on $\Omega^*$;\\
(vii) $g(s,t)\geq 0$ on $\overline{Q_T^*}$.\\
Then we have $f\leq g$ on $\overline{Q_T^*}$.
\end{Prop}

\begin{proof}[Proof of Theorem \ref{thm:main2}]

Note that it follows from Proposition \ref{eq:k's equation} that
\begin{equation*}
    \partial_t k - 4\pi s \partial_s^2 k - k \partial_s k \le 0, \quad k|_{s=0}=0, \quad k|_{s=\infty}=M,
\end{equation*}
and 
\begin{equation*}
    \partial_s k|_{s=\infty}=n^*(\infty,t)=0,\quad  \partial_s k \ge 0.
\end{equation*}

\noindent  {\bf Step I: $M < 8\pi$.} Define an auxiliary function as follows:
\ben\label{eq:auxiliary-g}
g(s):=\frac{8\pi q s}{1 + q s}\geq 0,\een
which implies that
\beno
- 4\pi s \partial_s^2 g - g \partial_s g = 0.
\eeno
Then
there exists $ q > 0$ such that $k(s,0) \le g(s)$ for any $ s \ge 0$. In fact, by choosing $s_0>0$ and $q$ large enough we have
\ben\label{eq:ks1}
\frac{k(s,0)}{s} \le \|n_{in} \|_{L^\infty(\mathbb{R}^2)}\leq \frac{8\pi q }{1 + qs },\quad 0<s\leq s_0,
\een
and
\ben\label{eq:ks2}
{k(s,0)} \le M  \leq \frac{8\pi qs }{1 + qs },\quad s\geq s_0,
\een
where the sharp value $(s_0, q_0)$ is the intersection point of the curve $L_1: q=\frac{\|n_{in} \|_{L^\infty} }{8\pi-\|n_{in} \|_{L^\infty}s}$ and $L_2: q=\frac{M}{(8\pi-M)s}$ shown as follows. More precisely if $s_0=\frac{M}{\|n_{in} \|_{L^\infty}}$, 
$q=q_0=\frac{\|n_{in} \|_{L^\infty} }{8\pi-M}$, then we have $\frac{8\pi q_0 }{1 + q_0s_0}=\|n_{in} \|_{L^\infty(\mathbb{R}^2)}$, 
$\frac{8\pi q_0s_0}{1 + q_0s_0}=M$, which implies \eqref{eq:ks1}, \eqref{eq:ks2}.

\begin{tikzpicture}[scale=1.3] 

    \draw[thick, ->] (-0.5,0) -- (6.5,0) node[right] {$s$};
    \draw[thick, ->] (0,-0.5) -- (0,6.5) node[above] {$q$};
    \node[below left] at (0,0) {$0$};


    \draw[blue, thick, domain=0:4.2, samples=100] 
        plot (\x, {5 / (5 - \x)}) 
        node[right, scale=0.9] {$L_1$};

    \draw[red, thick, domain=1.2:6.0, samples=100] 
        plot (\x, {7.5 / \x}) 
        node[above right, scale=0.9] {$L_2$};

    \coordinate (P) at (3, 2.5); 
    \fill[black] (P) circle (2pt);
    \node[above right, yshift=2pt] at (P) {$(s_0, q_0)$};

    \draw[dashed, gray, thick] (3,0) -- (3,2.5);
    \draw[dashed, gray, thick] (0,2.5) -- (3,2.5);

    \node[below, yshift=-2pt] at (3,0) {$s_0$};
    \node[left, xshift=-2pt] at (0,2.5) {$q_0$};

    \fill[black] (0,1) circle (1.5pt);
    \node[left, xshift=-2pt, scale=0.9] at (0,1) {$\frac{\|n_{in}\|_{L^\infty}}{8\pi}$};

\end{tikzpicture}

Hence take $q\geq q_0=\frac{\|n_{in}\|_{L^\infty}}{8\pi - M}$ and we have 
\ben\label{eq:ks0}
k(s,0) \le g(s), \quad 0\leq s\leq s_0= \frac{M}{\|n_{in}\|_{L^\infty}}.
\een
Moreover, 
by \eqref{eq:ks0} we have
\beno
{k(s,t)} \le M \leq g(s), \quad s\geq s_0.\eeno
Then  by Proposition \ref{lem:dn1995} in the domain $(0,s_0)\times(0,T)$ we have
\begin{equation*}
    k(s,t) \le \frac{8\pi q s}{1 + q s} ,
\end{equation*}
which implies 
\ben\label{eq:bound-Mless8pi}\|n (\cdot,t)\|_{L^\infty(\mathbb{R}^2)}=n^*(0,t)=\partial_s k|_{s=0}=  \lim_{s\rightarrow 0} \frac{k(s,t)}{s}\le  8\pi q= \frac{8\pi\|n_{in}\|_{L^\infty}}{8\pi - M},
\een
for any $t<T$ and specially $q$ is independent of $t.$

\noindent {\bf Step II:  $M = 8\pi$.} Take an auxiliary function \beno\phi(s,t) = (2T - t)^{-2} e^{-\frac{s}{4\pi(2T - t)}},\eeno which satisfies 
\ben\label{eq:phi-eq}
    \partial_t \phi + \partial_s^2 (4\pi s \phi) = 0.
\een
Furthermore, there holds
\begin{align*}
    \partial_t \int_0^\infty k \phi ds &= \int_0^\infty (\partial_t k \phi + k \partial_t \phi)ds \leq  \int_0^\infty [(4\pi s \partial_s^2 k + k \partial_s k) \phi + k \partial_t \phi] ds\\
    &= \int_0^\infty \left[ k \partial_s^2 (4\pi s \phi) - \frac{k^2}{2} \partial_s \phi + k \partial_t \phi \right]ds \\
    &= -\frac{1}{2} \int_0^\infty k^2 \partial_s \phi ds= \frac{1}{8\pi(2T - t)} \int_0^\infty k^2 \phi ds\le \frac{M}{8\pi(2T - t)} \int_0^\infty k \phi ds,
\end{align*}
which implies 
\begin{equation*}
    \partial_t \int_0^\infty (2T - t) k \phi ds\le 0 
\end{equation*}
due to  $M = 8\pi$. Let
\beno
R(t)=\int_0^\infty (2T - t) k \phi ds
\eeno   then
\begin{equation*}
    \frac{d}{dt} R(t)\leq0 ,\quad  R(t) \leq R(0) ,\quad \forall\ 0\le t\le T. 
\end{equation*}
Using $\frac{k(s,0)}{s} \le \|n_{in} \|_{L^\infty(\mathbb{R}^2)}$ and $k(s,0)\leq M=8\pi$ we have
\begin{align*}
    R(0) &= \int_0^\infty (2T) k(s,0) \phi(s,0) ds= (2T)^{-1} \int_0^\infty e^{-\frac{s}{8\pi T}} k(s,0) \, ds \\
    &\leq (2T)^{-1}\int_0^{\infty}e^{-\frac{s}{8\pi T}}\min\{\|n_{in} \|_{L^\infty}s,8\pi\} ds\\
    &= (2T)^{-1}\int_0^{\infty}8\pi Te^{-\frac{s}{8\pi T}}\partial_s\min\{\|n_{in} \|_{L^\infty}s,8\pi\} ds\\
    &=4\pi\|n_{in} \|_{L^\infty}\int_0^{{8\pi}/\|n_{in} \|_{L^\infty}}e^{-\frac{s}{8\pi T}}ds=4\pi\|n_{in} \|_{L^\infty}8\pi T(1-e^{-\frac{1}{\|n_{in} \|_{L^\infty} T}})
  \end{align*}
Consequently, we have
\begin{align*}
    &R(0)\geq R(t) =(2T - t)^{-1} \int_{0}^\infty e^{-\frac{s}{4\pi(2T - t)}} k(s, t) \, ds\\ \ge& (2T - t)^{-1} \int_{s_0}^\infty e^{-\frac{s}{4\pi(2T - t)}} k(s_0, t) \, ds = 4\pi e^{-\frac{s_0}{4\pi(2T - t)}} k(s_0, t)\ge4\pi e^{-\frac{s_0}{4\pi T }} k(s_0, t),
\end{align*}
which implies
\begin{equation*}
   { k(s_0, t) \le\frac{R(0)}{4\pi} e^{\frac{s_0}{4\pi T}} \le \|n_{in} \|_{L^\infty}8\pi T(1-e^{-\frac{1}{\|n_{in} \|_{L^\infty} T}})e^{\frac{s_0}{4\pi T}},} \quad \forall~ 0 \le t \le T.
\end{equation*}
Let $a=\frac{1}{\|n_{in} \|_{L^\infty} T}$, $b=\frac{s_0}{4\pi T}$ then we have \begin{equation}\label{b1}
   { k(s_0, t) \le 8\pi (1-e^{-a})e^{b}/a,} \quad \forall~ 0 \le t \le T,\quad b>0,\quad s_0=4\pi Tb.
\end{equation}
Now we claim that\begin{align}\label{b2}
    &\frac{1-e^{-a}}{a}e^{\frac{a}{a+4}}\le\frac{4}{a+4},\quad\forall~ a>0.
\end{align}In fact let $F(a)=\frac{1-e^{-a}}{a}e^{\frac{a}{a+4}}\frac{a+4}{4}$ then for $a>0$ we have $F(a)>0$ and\begin{align*}
    &\frac{F'(a)}{F(a)}=\frac{1}{e^a-1}-\frac{1}{a}+\frac{4}{(a+4)^2}+\frac{1}{a+4}=\frac{1}{e^a-1}-\frac{4^2}{a(a+4)^2}<0,
\end{align*}as $e^a-1>a+a^2/2+a^3/6>a(1+a/4)^2 $. Thus $F$ is decreasing for $a>0$ and\begin{align*}
    &{F(a)}\leq \lim_{a\to0+}F(a)=\lim_{a\to0+}\frac{1-e^{-a}}{a}=1,\quad \forall~ a>0,
\end{align*}which implies \eqref{b2}. Let $b=\frac{a}{a+4}$, $s_0=4\pi Tb$, $q=\frac{4}{as_0}=\frac{1}{\pi Tab}=\frac{a+4}{\pi Ta^2}$, $g(s)=\frac{8\pi qs}{1+qs}$, then $- 4\pi s \partial_s^2 g - g \partial_s g = 0$,  by  \eqref{b1} and \eqref{b2} we have\begin{equation}\label{b3}
   { k(s_0, t) \le 8\pi \frac{1-e^{-a}}{a}e^{\frac{a}{a+4}}\le8\pi \frac{4}{a+4}=\frac{8\pi qs_0}{1+qs_0}=g(s_0),\quad \forall~ 0 \le t \le T.} 
\end{equation}
Moreover, for $0<s<s_0$ we have
\beno
\frac{k(s,0)}{s}\leq \|n_{in} \|_{L^\infty}\leq 8 \|n_{in} \|_{L^\infty} =\frac{8\pi q}{1+qs_0}\leq \frac{8\pi q}{1+qs}.
\eeno
Then  by Proposition \ref{lem:dn1995} in the domain $(0,s_0)\times(0,T)$ we have
\begin{equation*}
    k(s,t) \le g(s)=\frac{8\pi q s}{1 + q s} ,
\end{equation*}
which implies (recall that $a=\frac{1}{\|n_{in} \|_{L^\infty} T}$)
\ben\label{eq3}&\|n (\cdot,t)\|_{L^\infty(\mathbb{R}^2)}=n^*(0,t)=\partial_s k|_{s=0}=\lim_{s\rightarrow 0} \frac{k(s,t)}{s}\le  8\pi q= \frac{8(a+4)}{Ta^2}\\&=8\|n_{in} \|_{L^\infty}(1+4T\|n_{in} \|_{L^\infty}),\nonumber
\een
for any $t<T$. 
This completes the proof.
\end{proof}

\section{Monotonicity property  and proof of Proposition \ref{eq:k's equation}}

In this section, we are committed to proving Proposition \ref{eq:k's equation}.

Before the proof of  Proposition \ref{eq:k's equation}, we give the following useful lemma.
\begin{Lem}
There hold
\ben\label{eq:mr}
\int_{\{x\in \mathbb{R}^2; n(x,t)>s\}}\partial_t n(x,t) dx=\frac{\partial k}{\partial t}(\mu(s,t),t ),
\een
and
\ben\label{eq:mr-cor}
k(\mu(s,t),t)=\int_{\{x\in \mathbb{R}^2; n(x,t)>s\}} n(x,t) dx.
\een
\end{Lem}

\begin{proof}
It follows from a theorem of Mossino and Rakotoson ((2.12), \cite{MR1986}). Here we give a complete proof.

\textbf{Proof of \eqref{eq:mr-cor}.} 
First, by the classical theory of decreasing rearrangements, the function $n^*(\sigma, t)$ is equimeasurable with $n(x,t)$. This implies that the integral of $n(x,t)$ over any of its super-level sets is exactly equal to the integral of $n^*(\sigma, t)$ over the corresponding interval $[0, \mu(s,t)]$, where $\mu(s,t)$ is the Lebesgue measure of the super-level set $\{x \in \mathbb{R}^2 \mid n(x,t) > s\}$. Therefore, we have
\begin{equation*}
    \int_{\{n(x,t)>s\}} n(x,t) \,dx = \int_0^{\mu(s,t)} n^*(\tau,t) \,d\tau.
\end{equation*}
Recalling the definition of $k(s,t)$ in \eqref{eq:k-def}, the right-hand side is exactly $k(\mu(s,t), t)$. This proves \eqref{eq:mr-cor}.

\textbf{Proof of \eqref{eq:mr}.} 
We introduce an auxiliary function $G(t)$ representing the excess mass above the threshold level $s$:
\begin{equation}\label{eq:G_def}
    G(t) = \int_{\{n(x,t)>s\}} \big( n(x,t) - s \big) \,dx.
\end{equation}
First, differentiating \eqref{eq:G_def} directly with respect to $t$ in the physical space yields:
\begin{equation}\label{eq:dt_G_1}
    \frac{d}{dt} G(t) = \int_{\{n(x,t)>s\}} \partial_t n(x,t) \,dx.
\end{equation}
Notice that the boundary terms arising from the differentiation of the domain of integration vanish, since $n(x,t) = s$ on the boundary $\partial \{n(x,t)>s\}$.

Second, we rewrite $G(t)$ using the rearrangement space identity \eqref{eq:mr-cor}:
\begin{align*}
    G(t) &= \int_{\{n(x,t)>s\}} n(x,t) \,dx - \int_{\{n(x,t)>s\}} s \,dx \\
    &= k(\mu(s,t), t) - s \mu(s,t).
\end{align*}
Now, differentiating this expression with respect to $t$ using the chain rule, we obtain:
\begin{equation}\label{eq:dt_G_2}
    \frac{d}{dt} G(t) = \frac{\partial k}{\partial \mu}(\mu(s,t), t) \frac{\partial \mu}{\partial t} + \frac{\partial k}{\partial t}(\mu(s,t), t) - s \frac{\partial \mu}{\partial t}.
\end{equation}
By the Fundamental Theorem of Calculus applied to definition \eqref{eq:k-def}, we know that $\frac{\partial k}{\partial \mu} = n^*(\mu, t)$. Substituting $\mu = \mu(s,t)$ and utilizing the critical identity \eqref{eq:s-n}, which states $n^*(\mu(s,t), t) = s$, we have
\begin{equation*}
    \frac{\partial k}{\partial \mu}(\mu(s,t), t) = s.
\end{equation*}
Plugging this back into \eqref{eq:dt_G_2}, the terms involving $\frac{\partial \mu}{\partial t}$ perfectly cancel out:
\begin{align*}
    \frac{d}{dt} G(t) &= s \frac{\partial \mu}{\partial t} + \frac{\partial k}{\partial t}(\mu(s,t), t) - s \frac{\partial \mu}{\partial t} \\
    &= \frac{\partial k}{\partial t}(\mu(s,t), t),
\end{align*}
which and \eqref{eq:dt_G_1} implies \eqref{eq:mr}.
    
\end{proof}

\begin{proof}[Proof of Proposition \ref{eq:k's equation}]

The proof is divided into three steps.

\textbf{Step I: The energy estimate.}
For $t\in (\tau,T)$, $h>0$ and $\rho\in (0, n^*(0,t))$, let
\begin{align} \label{eq:test-1}\,\, \Phi(s)=\left\{
\begin{aligned}
0,\quad  &s\leq \rho,\\
s-\rho,\quad  & \rho<s\leq \rho+h,\\
h,\quad  & s>\rho+h.
\end{aligned}
\right. \end{align}
Multiplying $ \Phi(n)$ on both sides of $\eqref{ksns}_1$, we have
\ben\label{eq:ks-energy}
\int_{\mathbb{R}^2}\frac{\partial}{\partial t}n  \Phi(n)dx&=& -\int_{\mathbb{R}^2}\nabla n  \cdot \nabla \Phi(n)dx+\int_{\mathbb{R}^2}n \nabla c  \cdot \nabla \Phi(n)dx\nonumber\\
&&+ \int_{\mathbb{R}^2} n  u\cdot \nabla \Phi(n)dx=:I_1+{\color{red}I_2}+I_3.
\een
For the first term $I_1$ of the right hand of \eqref{eq:ks-energy}, by \eqref{eq:test-1} we get
\ben\label{eq:laplace n}
&&\lim_{h\searrow 0}\frac{1}{h}\int_{\mathbb{R}^2}\nabla n  \cdot \nabla \Phi(n)dx\nonumber\\
&=& \lim_{h\searrow 0}\frac{1}{h}\left(\int_{n>\rho}|\nabla n|^2  dx-\int_{n>\rho+h}|\nabla n|^2  dx\right)\nonumber\\
&=&-\frac{\partial}{\partial \rho}\int_{n>\rho}|\nabla n|^2  dx.
\een
For the second term $I_2$  of the right hand of \eqref{eq:ks-energy}, let
\begin{align} \label{eq:test-2}\,\, \Psi(s)=\left\{
\begin{aligned}
0,\quad  &s\leq \rho,\\
\frac12(s^2-\rho^2),\quad  & \rho<s\leq \rho+h,\\
h(\rho+\frac{h}{2}),\quad  & s>\rho+h.
\end{aligned}
\right. \end{align}
Then \eqref{eq:test-1}, \eqref{eq:test-2} and integration by parts imply 
\beno
\int_{\mathbb{R}^2}n \nabla c  \cdot \nabla \Phi(n)dx=\int_{\mathbb{R}^2} \nabla c  \cdot \nabla \Psi(n)dx=\int_{\mathbb{R}^2}n  \Psi(n)dx.
\eeno
Consequently, using \eqref{eq:test-2} again we have
\ben\label{eq:n nablac}
&&\lim_{h\searrow 0}\frac{1}{h}\int_{\mathbb{R}^2}n \nabla c  \cdot \nabla \Phi(n)dx\nonumber\\
&=& \lim_{h\searrow 0}\frac{1}{2h}\int_{\rho<n\leq \rho+h} n(n^2-\rho^2)dx+\int_{n> \rho+h}n (\rho+\frac{h}{2})dx\nonumber\\
&=&\rho \int_{n> \rho}ndx\nonumber\\
&= &n^*(\mu(\rho,t),t)\int_{n> \rho}n dx\nonumber\\
&=& \frac{\partial k}{\partial s}(\mu(\rho,t),t) k(\mu(\rho,t),t),
\een
where we used \eqref{eq:s-n} and \eqref{eq:mr-cor}.\\
For the term $I_3$, there holds
\ben\label{eq:u term}
\int_{\mathbb{R}^2} n  u\cdot \nabla \Phi(n)dx=\int_{\mathbb{R}^2} u  \cdot \nabla \Psi(n)dx=0.
\een
Moreover, by \eqref{eq:mr} we have 
\ben\label{eq:partial t n}
\lim_{h\searrow 0}\frac{1}{h}\int_{\mathbb{R}^2}\frac{\partial}{\partial t}n  \Phi(n)dx=\int_{n>\rho}\partial_t ndx=\frac{\partial k}{\partial t}(\mu(\rho,t),t).
\een

Consequently, recalling \eqref{eq:ks-energy} and concluding the above computations such as \eqref{eq:laplace n}, \eqref{eq:n nablac}, \eqref{eq:u term} and \eqref{eq:partial t n},  we have 
\ben\label{eq:energy-k}
\frac{\partial k}{\partial t}(\mu(\rho,t),t)-\frac{\partial}{\partial \rho}\int_{n>\rho}|\nabla n|^2  dx= \frac{\partial k}{\partial s}(\mu(\rho,t),t) k(\mu(\rho,t),t).
\een

\textbf{Step II: The lower bound of the diffusion term.} Let
\beno
 E(\rho):= \int_{n>\rho}|\nabla n|^2  dx, \eeno
and we claim that 
\ben\label{eq:bound of mu}
  -\frac{\partial}{\partial \rho} E(\rho)\geq 4\pi \mu\left( -\frac{\partial\mu}{\partial\rho}\right)^{-1}.
\een
\textbf{Proof of \eqref{eq:bound of mu}.}  For any integrable function $f$, by  Co-area formula (for example Theorem 3.11 \cite{EG2015}) we have
\beno
\int_{\mathbb{R}^2} f(x) |\nabla n| \, dx = \int_{-\infty}^{\infty} \left( \int_{\Gamma_s} f(x) \, d\mathcal{H}^1 \right) ds,
\eeno
where $\mathcal{H}^1$ denotes the 1-dimensional Hausdorff measure (arc length) and the level set $\Gamma_s=\{x; n(x,t)=s\}$. Specifically, for integrals over the super-level set $n > \rho$ we get
\ben\label{eq:co-area}
\int_{n > \rho} g(x) \, dx = \int_{\rho}^{\infty} \left( \int_{\Gamma_s} \frac{g(x)}{|\nabla n|} \, d\mathcal{H}^1 \right) ds.
\een
Taking $g(x) = 1$ in \eqref{eq:co-area}, then we have
\beno
\mu(\rho) = \int_{\rho}^{\infty} \left( \int_{\Gamma_s} \frac{1}{|\nabla n|} \, d\mathcal{H}^1 \right) ds,
\eeno
which yields that 
\ben \label{eq:derivative of mu}
-\frac{\partial \mu}{\partial \rho} = \int_{\Gamma_\rho} \frac{1}{|\nabla n|} \, d\mathcal{H}^1.
\een
Moreover, choosing $g(x) = |\nabla n|^2$ in \eqref{eq:co-area} we have
\beno
E(\rho) = \int_{\rho}^{\infty} \left( \int_{\Gamma_s} |\nabla n| \, d\mathcal{H}^1 \right) ds,
\eeno
which implies that
\ben \label{eq:dE}
-\frac{\partial E}{\partial \rho} = -\frac{\partial}{\partial \rho}\int_{n>\rho}|\nabla n|^2 \, dx = \int_{\Gamma_\rho} |\nabla n| \, d\mathcal{H}^1.
\een

The classical isoperimetric inequality in $\mathbb{R}^2$ states that for a domain with area $A$ and perimeter $L$, $4\pi A \leq L^2$ (see, for example,  \cite{O1978}).
In our context, $A = \mu(\rho)$ and perimeter $= L(\rho)$. Therefore:
\begin{equation} \label{eq:iso}
4\pi \mu(\rho) \leq L(\rho)^2,
\end{equation}
where $L(\rho) = \mathcal{H}^1(\Gamma_\rho)$ be the length (perimeter) of the level set $\Gamma_\rho$. Then
\beno
L(\rho) = \int_{\Gamma_\rho} 1 \, d\mathcal{H}^1 = \int_{\Gamma_\rho} \sqrt{|\nabla n|} \cdot \frac{1}{\sqrt{|\nabla n|}} \, d\mathcal{H}^1.
\eeno
Applying the Cauchy-Schwarz inequality, 
 (\ref{eq:derivative of mu}) and (\ref{eq:dE}), we obtain that 
\beno \label{eq:CS}
L(\rho)^2 \leq \left( -\frac{\partial E}{\partial \rho} \right) \cdot \left( -\frac{\partial \mu}{\partial \rho} \right),
\eeno
which and \eqref{eq:iso} implies the claim \eqref{eq:bound of mu}.

\textbf{Step III: Conclusions.}
By \eqref{eq:energy-k} and \eqref{eq:bound of mu} we have 
\beno
4\pi \mu \left( -\frac{\partial\mu}{\partial\rho}\right)^{-1}\leq -\frac{\partial k}{\partial t}(\mu(\rho,t),t)+\frac{\partial k}{\partial s}(\mu(\rho,t),t) k(\mu(\rho,t),t).
\eeno
By \eqref{eq:s-n} and \eqref{eq:k-def}, this implies 
\beno
-\frac{\partial n^*}{\partial s}(s,t)=-\frac{\partial^2 k}{\partial s^2}(s,t)\leq \frac{1}{4\pi s}\left[-\frac{\partial k}{\partial t}(s,t)+\frac{\partial k}{\partial s}(s,t) k(s,t)\right]
\eeno
The proof of \eqref{eq:mono} is complete and \eqref{eq:mono-b} follows from \eqref{eq:k-def} and the property of $n^*$.
\end{proof}

\section{An improved maximum principle and proof of Proposition \ref{lem:dn1995}}

In this section, we are committed to proving
Proposition \ref{lem:dn1995}.

\begin{proof}[Proof of Proposition \ref{lem:dn1995}]

Let $w(s, t) = f(s, t) - g(s, t)$. Next we  prove that $w(s, t) \le 0$ almost everywhere in $Q_T^*$.
Subtracting the inequality for $g$ from the inequality for $f$ in condition (iii), we obtain the differential inequality for $w$:
\begin{equation} \label{eq:diff}
    \frac{\partial w}{\partial t} - 4\pi s \frac{\partial^2 w}{\partial s^2} - \left( f \frac{\partial f}{\partial s} - g \frac{\partial g}{\partial s} \right) \le 0.
\end{equation}
Define the positive part of the difference as $w_+(s, t) = \max\{0, w(s, t)\}$. Assume  that $w_+ \in H^1(0, \Lambda)$ for a.e. $t$ without loss of generality, otherwise we can use an approximation method. Multiplying the inequality by $w_+$ and integrating over the spatial domain $\Omega^* = (0, \Lambda)$, we get:
\begin{equation} \label{eq:integral}
    \frac{1}{2} \frac{d}{dt} \int_0^{\Lambda} (w_+)^2 \, ds - 4\pi \int_0^{\Lambda} s \frac{\partial^2 w}{\partial s^2} w_+ \, ds -  \int_0^{\Lambda} \left( f \frac{\partial f}{\partial s} - g \frac{\partial g}{\partial s} \right)w_+ \, ds \le 0.
\end{equation}

\noindent \textbf{Step I. The Diffusion term.}
We perform integration by parts on the second term. Note that $w_+$ is non-zero only where $w > 0$, so $\partial_s w = \partial_s w_+$ on the support of $w_+$.
\begin{align*}
    - \int_0^{\Lambda} s \frac{\partial^2 w}{\partial s^2} w_+ \, ds &= - \left[ s \frac{\partial w}{\partial s} w_+ \right]_0^{\Lambda} + \int_0^{\Lambda} \frac{\partial}{\partial s} (s w_+) \frac{\partial w}{\partial s} \, ds \\
    &= - \Lambda \frac{\partial w}{\partial s}(\Lambda, t) w_+(\Lambda, t) + \int_0^{\Lambda} \left( w_+ \frac{\partial w_+}{\partial s} + s \left(\frac{\partial w_+}{\partial s}\right)^2 \right) \, ds.
\end{align*}
From one condition of (v), $\frac{\partial f}{\partial s}(\Lambda, t) \le \frac{\partial g}{\partial s}(\Lambda, t)$, which implies $\frac{\partial w}{\partial s}(\Lambda, t) \le 0$. Since $w_+ \ge 0$, the boundary term $- \Lambda \frac{\partial w}{\partial s} w_+$ is non-negative. The other condition of (v) $f(\Lambda, t) \leq g(\Lambda, t) $ implies $- \Lambda \frac{\partial w}{\partial s} w_+=0$ if $\Lambda<\infty$.
Also, 
\beno
\int_0^{\Lambda} w_+ \frac{\partial w_+}{\partial s} \, ds = \frac{1}{2} (w_+(\Lambda, t))^2 \ge 0
\eeno 
due to the condition (iv).
Thus:
\ben\label{eq: bound of diffusion}
    - 4\pi \int_0^{\Lambda} s \frac{\partial^2 w}{\partial s^2} w_+ \, ds \ge 4\pi \int_0^{\Lambda} s \left(\frac{\partial w_+}{\partial s}\right)^2 \, ds.
\een

\noindent \textbf{Step II. The convection term.}
Note that the third term in (\ref{eq:integral})
satisfies
\beno
     \int_0^{\Lambda} w_+\left(f\frac{\partial}{\partial s}f-g\frac{\partial}{\partial s}g\right) \, ds
     &=&  \int_0^{\Lambda} w_+^2\frac{\partial}{\partial s} f+g\frac{\partial w_+}{\partial s} w_+\, ds\\
    &\leq &C_1(t) \int_0^{\Lambda} w_+^2ds+\int_0^{\Lambda} \left| f\frac{\partial w_+}{\partial s}\right| w_+\, ds
\eeno
where we used the conditions (ii), (vii) and $g\leq f$ when $w_+$ does not vanish. Using conditions (ii) and (iv) we have
\beno
|f|\leq \min\{C_1(t), C_1(t)s\}\leq C_1(t)\sqrt{s},
\eeno
which implies
\beno
\int_0^{\Lambda} \left| f\frac{\partial w_+}{\partial s}\right| w_+\, ds\leq \int_0^{\Lambda} s (\partial_s w_+)^2+ C_1(t)^2 \int_0^{\Lambda} w_+^2ds.\eeno
Hence, we have
\ben\label{eq:third term}
     \int_0^{\Lambda} w_+\left(f\frac{\partial}{\partial s}f-g\frac{\partial}{\partial s}g\right) \, ds 
    &\leq &\left(C_1(t) + C_1(t)^2  \right)\int_0^{\Lambda} w_+^2ds+\int_0^{\Lambda} s(\partial_s w_+)^2 ds\nonumber\\
\een


Combining the estimates \eqref{eq: bound of diffusion} and \eqref{eq:third term} into (\ref{eq:integral}), we have
\[
    \frac{1}{2} \frac{d}{dt} \|w_+\|_{L^2}^2 + 4\pi \int_0^{\Lambda} s (\partial_s w_+)^2 \le 2\pi \int_0^{\Lambda} s (\partial_s w_+)^2 + \left(C_1(t) + C_1(t)^2  \right) \int_0^{\Lambda} (w_+)^2.
\]
Absorbing the gradient term:
\[
    \frac{d}{dt} \|w_+(t)\|_{L^2}^2 \le 2C \|w_+(t)\|_{L^2}^2.
\]
From the condition (vi), $w(s, 0) \le 0$, so $w_+(s, 0) = 0$, implying $\|w_+(0)\|_{L^2}^2 = 0$. By Gronwall's inequality, $\|w_+(t)\|_{L^2}^2 = 0$ for all $t \in [0, T]$.
Therefore, $w_+(s, t) = 0$ a.e., which implies $f(s, t) \le g(s, t)$ on $Q_T^*$.
\end{proof}




\section{Appendix: Local well-posed result of the system \eqref{ksns}}

In this section, we establish the local existence and uniqueness of mild solutions to the system (1.1) using the Contraction Mapping Principle.

	\begin{Thm}
			\label{thm:local} Assume that $\phi\in L^\infty\left([0,\infty); \dot{W}^{1,\infty}(\mathbb{R}^{2})\right)$, $0\leq n_{in}\in L^1{\cap L^{\infty}}(\mathbb{R}^{2})$ and $u_{in}\in W^{1,2}(\mathbb{R}^{2})$. Then there exists a local strong 
			 solution $(n, u, \pi)$ to the Keller-Segel-Navier-Stokes equations \eqref{ksns} with the initial data $(n_{in},u_{in})$ .
		\end{Thm}
\begin{proof}
Let $P$ denote the Leray projection operator onto the divergence-free vector fields in $L^2(\mathbb{R}^2)$. The system (1.1) can be rewritten as the following integral equations:
\begin{equation}\label{eq:mild_form}
\left\{
\begin{aligned}
n(t) &= e^{t\Delta} n_{in} - \int_0^t e^{(t-s)\Delta} \nabla \cdot (n u + n \nabla c)(s) \, ds, \\
u(t) &= e^{t\Delta} u_{in} - \int_0^t e^{(t-s)\Delta} P \left( u \cdot \nabla u + n \nabla \phi \right)(s) \, ds,
\end{aligned}
\right.
\end{equation}
where $c(t) = (-\Delta)^{-1} n(t) = -\frac{1}{2\pi} \ln|\cdot| * n(t)$.

\textbf{Step I: Function Spaces.}
We seek solutions in a Banach space that captures the critical scaling of the Keller-Segel system while maintaining the regularity required for the Navier-Stokes equations. For $T > 0$,  define the space $\mathcal{X}_T$ as:
\[
\mathcal{X}_T := \left\{ (n, u) \;\middle|\; 
\begin{aligned}
&n \in L^\infty(0, T; L^1{\cap L^{\infty}}(\mathbb{R}^2)),\\
&u \in L^\infty(0, T; H^1(\mathbb{R}^2)), 
\end{aligned} 
\right\}
\]
where the norm on $\mathcal{X}_T$ is defined by
\[
\|(n, u)\|_{\mathcal{X}_T} := \sup_{0<t<T}  \|n(t)\|_{L^1\cap L^{\infty}}  + \sup_{0<t<T} \|u(t)\|_{H^1}.
\]
Define a closed ball in $\mathcal{X}_T$ centered at the origin with radius $R$ with $T\leq T_1$:
\[
B_{R}(T) := \left\{ (n, u) \in \mathcal{X}_T \;\middle|\; \|(n, u)\|_{\mathcal{X}_T} \le R 
\right\},
\]
where  $R$ is chosen sufficiently large, e.g., $R > 2(\|n_{in}\|_{L^1\cap L^{\infty}} + \|u_{in}\|_{H^1})$.

{\bf Step II: Linear estimates and contraction mapping.}
Define the map\\ $\Phi(n, u) = (\Phi_1(n, u), \Phi_2(n, u))$ corresponding to the right-hand side of \eqref{eq:mild_form}. We aim to show that for sufficiently small $T$, $\Phi$ maps $B_{R}(T)$ into itself and is a contraction.

\textbf{ II.1: Estimates for the Cell Density $n$.}
Recall the $L^p$-$L^q$ smoothing estimates for the heat semigroup $e^{t\Delta}$ in $\mathbb{R}^2$ (see, for example,  Proposition A.16 in \cite{BV2022}):
\ben\label{eq:semigroup}
\|\nabla^k e^{t\Delta} f\|_{L^q} \le C t^{\frac{1}{q}-\frac{1}{p} - \frac{k}{2}} \|f\|_{L^p}, \quad 1 \le p \le q \le \infty.
\een
By the Hardy-Littlewood-Sobolev (HLS) inequality, since $n \in L^1\cap L^{\infty}\subset L^{4/3}$, 
\ben\label{eq:HLS}
\|\nabla c\|_{L^{4}}= \|\nabla (-\Delta)^{-1} n\|_{L^{4}} \leq 
C \|n\|_{L^{4/3}}\leq C\|n\|_{L^1\cap L^{\infty}}.
\een
For $\Phi_1$, by \eqref{eq:mild_form} and \eqref{eq:semigroup} we estimate the $L^{4/3}$ norm with the time weight $t^{1/4}$:
\begin{align*}
&\|\Phi_1(n, u)(t)\|_{L^1\cap L^{\infty}} \le \|e^{t\Delta} n_{in}\|_{L^1\cap L^{\infty}} + \int_0^t \left\| e^{(t-s)\Delta} \nabla \cdot (nu + n\nabla c)(s) \right\|_{L^1\cap L^{\infty}} ds \\
\le& \| n_{in}\|_{L^1\cap L^{\infty}} + C \int_0^t ((t-s)^{-1/2}+(t-s)^{-3/4}) \left( \|nu\|_{L^1\cap L^4} + \|n\nabla c\|_{L^1\cap L^4} \right) ds.
\end{align*}
Using  H\"{o}lder's inequality, \eqref{eq:HLS} and  Gagliardo-Nirenberg interpolation inequality
 we have
\[
\|n \nabla c\|_{L^1\cap L^4} \le \|n\|_{L^1\cap L^{\infty}} \|\nabla c\|_{L^4} \le C\|n\|_{L^1\cap L^{\infty}}^2,
\]
\[
\|n u\|_{L^1\cap L^4} \le C \|n\|_{L^1\cap L^{\infty}} \|u\|_{L^4} \le C \|n\|_{L^1\cap L^{\infty}}\|u\|_{H^1}.
\]
Then,
\begin{align*}
\|\Phi_1(t)\|_{L^1\cap L^{\infty}} &\le \| n_{in}\|_{L^1\cap L^{\infty}}+ C \int_0^t ((t-s)^{-1/2}+(t-s)^{-3/4})  \|(n,u)\|_{\mathcal{X}_T}^2 \, ds \\
&\leq \| n_{in}\|_{L^1\cap L^{\infty}}+ C (t^{1/2}+t^{1/4})  \|(n,u)\|_{\mathcal{X}_T}^2.
\end{align*}
Thus, for $t \in (0, T)$,  we have
\[
\sup_{0<t<T} \|\Phi_1(t)\|_{L^1\cap L^{\infty}} \le  \| n_{in}\|_{L^1\cap L^{\infty}}+ C(T^{1/2}+ T^{1/4})R^2\le  \| n_{in}\|_{L^1\cap L^{\infty}}+R/4,
\]
if $T$ is sufficiently small.

\textbf{II.2: Estimates for the Velocity $u$.}
For $\Phi_2$, we need to bound the $H^1$ norm.
\[
\|\Phi_2(t)\|_{H^1} \le \|e^{t\Delta} u_{in}\|_{H^1} + \int_0^t \|e^{(t-s)\Delta} P \nabla \cdot (u \otimes u)\|_{H^1} ds + \int_0^t \|e^{(t-s)\Delta} P (n \nabla \phi)\|_{H^1} ds.
\]
For the convection term, 
 using \eqref{eq:semigroup} again we have
\begin{align*}
\int_0^t \|\nabla e^{(t-s)\Delta} (u \otimes u)\|_{L^2} ds &\le C \int_0^t (t-s)^{-1/2} \|u\|_{L^4}^2 ds \\
&\le C T^{1/2} \sup_t \|u(t)\|_{H^1}^2,
\end{align*}
and
\begin{align*}
\int_0^t \|\nabla^2 e^{(t-s)\Delta} (u \otimes u)\|_{L^2} ds &\le C \int_0^t (t-s)^{-2/3} \|u\nabla u\|_{L^\frac32} ds \\
&\le C T^{1/3} \sup_t \|u(t)\|_{H^1}^2,
\end{align*}

For the coupling term $n \nabla \phi$, since $\phi \in \dot{W}^{1,\infty}$:
\begin{align*}
\int_0^t \|e^{(t-s)\Delta} (n \nabla \phi)\|_{H^1} ds &\leq \int_0^t \| e^{(t-s)\Delta} (n \nabla \phi)\|_{L^2} ds+\int_0^t \|\nabla e^{(t-s)\Delta} (n \nabla \phi)\|_{L^2} ds \\
&\le C \int_0^t (1+(t-s)^{-1/2}) \|n\|_{L^2} \|\nabla \phi\|_{L^\infty} ds\\
&\le C(t +t^{1/2})\sup_t\|n\|_{L^1\cap L^{\infty}} \|\nabla \phi\|_{L^\infty}.
\end{align*}
Now we have
\[
\sup_{0<t<T} \|\Phi_2(t)\|_{H^1} \le  \| u_{in}\|_{H^1}+ C(T+ T^{1/3})(R^2+R\|\nabla \phi\|_{L^{\infty}((0,T),L^\infty)})\le  \| u_{in}\|_{H^1}+ R/4,
\] and 
\[\|\Phi(n,u)\|_{\mathcal{X}_T}\le \| n_{in}\|_{L^1\cap L^{\infty}}+\| u_{in}\|_{H^1}+R/2<R,\]
if $T$ 
is sufficiently small.

\textbf{Step III: Conclusion.}
Choose $T\in(0,1)$ small enough. Then $\Phi$ maps $B_{R}(T)$ into itself. Similar estimates on the difference $(\Phi(n_1, u_1) - \Phi(n_2, u_2))$ show that $\Phi$ is a contraction map. 
Thus, by the Banach Fixed Point Theorem, there exists a unique mild solution to (\ref{ksns}) on $[0, T)$. Moreover, $n\geq 0$ due to the maximum principle. Standard bootstrapping arguments allow us to upgrade this mild solution to a strong solution for $t > 0$.
\end{proof}
\begin{Rem}\label{6.2}
In the proof, $T\in(0,1)$ depends only on $ \| n_{in}\|_{L^1\cap L^{\infty}}$, $\| u_{in}\|_{H^1}$ and 
$\|\nabla \phi\|_{L^{\infty}((0,1),L^\infty)}$. Thus if the solution blows up in finite time $T^*$ then \\
$n \not\in L^\infty(0, T^*; L^1{\cap L^{\infty}}(\mathbb{R}^2))$ or $u \not\in L^\infty(0, T^*; H^1(\mathbb{R}^2))$.
\end{Rem}

\medskip
\noindent {\bf Acknowledgment:}\,
 W. Wang was supported by National Key R\&D Program of China (No. 2023YFA1009200) and NSFC under grant 12471219. Z. Zhang is supported by NSF of China under
Grant No. 12288101.

\medskip
\noindent\textbf{Data Availability Statement:}
Data sharing is not applicable to this article as no data sets were generated or analysed during the current study.

\noindent\textbf{Conflict of Interest:}
The authors declare that they have no conflict of interest.

	%

\end{document}